\documentclass[12pt]{amsart}
\usepackage{amsmath,amssymb,amsbsy,amsfonts,amsthm,latexsym,
                        amsopn,amstext,amsxtra,euscript,amscd,color,booktabs}
\usepackage[english]{babel}
\usepackage{url,enumerate}
\usepackage[colorlinks,linkcolor=blue,anchorcolor=blue,citecolor=blue]{hyperref}

\begin{document}

\newtheorem{theorem}{Theorem}
\newtheorem{corollary}[theorem]{Corollary}
\newtheorem{lemma}[theorem]{Lemma}
\newtheorem{algol}{Algorithm}
\newtheorem{cor}[theorem]{Corollary}
\newtheorem{prop}[theorem]{Proposition}

\theoremstyle{definition}
\newtheorem{remark}{Remark}

\newcommand{\comm}[1]{\marginpar{%
\vskip-\baselineskip 
\raggedright\footnotesize
\itshape\hrule\smallskip#1\par\smallskip\hrule}}

\def\cA{{\mathcal A}}
\def\cB{{\mathcal B}}
\def\cC{{\mathcal C}}
\def\cD{{\mathcal D}}
\def\cE{{\mathcal E}}
\def\cF{{\mathcal F}}
\def\cG{{\mathcal G}}
\def\cH{{\mathcal H}}
\def\cI{{\mathcal I}}
\def\cJ{{\mathcal J}}
\def\cK{{\mathcal K}}
\def\cL{{\mathcal L}}
\def\cM{{\mathcal M}}
\def\cN{{\mathcal N}}
\def\cO{{\mathcal O}}
\def\cP{{\mathcal P}}
\def\cQ{{\mathcal Q}}
\def\cR{{\mathcal R}}
\def\cS{{\mathcal S}}
\def\cT{{\mathcal T}}
\def\cU{{\mathcal U}}
\def\cV{{\mathcal V}}
\def\cW{{\mathcal W}}
\def\cX{{\mathcal X}}
\def\cY{{\mathcal Y}}
\def\cZ{{\mathcal Z}}

\def\C{\mathbb{C}}
\def\F{\mathbb{F}}
\def\K{\mathbb{K}}
\def\Z{\mathbb{Z}}
\def\R{\mathbb{R}}
\def\Q{\mathbb{Q}}
\def\N{\mathbb{N}}
\def\M{{\mathsf{M}}}

\def\({\left(}
\def\){\right)}
\def\[{\left[}
\def\]{\right]}
\def\<{\langle}
\def\>{\rangle}

\def\e{e}

\def\eq{\e_q}
\def\fS{{\mathfrak S}}

\def\lcm{{\mathrm{lcm}}\,}

\def\fl#1{\left\lfloor#1\right\rfloor}
\def\rf#1{\left\lceil#1\right\rceil}
\def\mand{\qquad\mbox{and}\qquad}

\def\jt{\tilde\jmath}
\def\ellmax{\ell_{\rm max}}
\def\llog{\log\log}

\def\Qbar{\overline{\Q}}
\def\GL{{\rm GL}}
\def\Aut{{\rm Aut}}
\def\End{{\rm End}}
\def\Gal{{\rm Gal}}
\def\tr{{\operatorname{tr}}}


\title[Distribution of Atkin and Elkies primes]{On the distribution of Atkin and Elkies primes for reductions of elliptic curves on average}

\author{Igor E.~Shparlinski} 
\address{Department of Pure Mathematics, University of 
New South Wales, Sydney, NSW 2052, Australia}
\email{igor.shparlinski@unsw.edu.au}

\author{Andrew V. Sutherland}  
\address{Department of Mathematics, Massachusetts Institute of Technology, Cambridge, Massachusetts 02139, USA} 
\email{drew@math.mit.edu}


\begin{abstract}
For an elliptic curve $E/\Q$ without complex multiplication we study the distribution of Atkin and Elkies primes $\ell$, on average, over all good reductions of $E$ modulo primes $p$.
We show that, under the Generalised Riemann Hypothesis, for almost all primes $p$ there are enough small Elkies primes $\ell$ to ensure that the Schoof-Elkies-Atkin point-counting algorithm runs in $(\log p)^{4+o(1)}$ expected time.
\end{abstract}

\maketitle

\section{Introduction}  
Let $E$ be a fixed elliptic curve over $\Q$ given by an integral Weierstrass model of minimal discriminant $\Delta_E$, and let $\F_p$ denote the finite field with $p$ elements.
Primes $p$ that do not divide $\Delta_E$ are said to be \emph{primes of good reduction} (for $E$), and for such primes $p$ we let $E_p$ denote the elliptic curve over $\F_p$ obtained by reducing the coefficients of $E$ modulo~$p$.
We  assume throughout that $E$ does not have \emph{complex multiplication} (CM), meaning that ${\rm End}_{\Qbar}(E)\simeq\Z$.
This assumption excludes only a finite set of $\Qbar$-isomorphism classes of elliptic curves for which the point-counting problem we consider is easily addressed in any case.
See~\cite{ACDFLNV,Silv} for background on  elliptic curves.

We always assume that $p$ is large enough,  and in particular, that~$p$ is a prime of good reduction greater than $3$.
We denote by $N_p$  the cardinality of $E_p(\F_p)$, the group of $\F_p$-rational points on $E_p$,
and define the {\it trace of Frobenius\/}
$t_p = p+1 - N_p$.
We say that an odd prime 
$\ell \ne p$ is an  {\it Elkies prime\/} for $E_p$
if the discriminant
$$
D_p=t_p^2 - 4p
$$ 
is a quadratic residue
modulo $\ell$; otherwise $\ell\ne p$ is called an {\it  Atkin prime\/} for $E_p$.
We note that the Hasse bound implies $t_p^2 < 4p$, so $D_p$ is always negative.

Recall that an elliptic curve over $\F_p$ is {\it ordinary\/} if its trace of Frobenius $t_p$ is not a multiple of $p$; for $p>3$ we can have $p\mid t_p$ only when $t_p=0$.
We therefore say that a prime $p$ is {\it ordinary\/} (for $E$) if $t_p\ne 0$,
and we say that  $p$ is {\it supersingular\/} otherwise.
It is well known that when $E$ does not have~CM almost all primes are ordinary;
in fact we know from the striking results of Elkies~\cite{Elk1} that while there are infinitely many supersingular primes,
the number of supersingular primes $p\le P$ is bounded by $O(P^{3/4})$.

The {\it Schoof-Elkies-Atkin algorithm\/} (SEA) is a widely used method to determine the number of rational points on an elliptic curve over a finite field.
For finite fields of large characteristic (in particular, the prime fields considered here), it is believed to be the asymptotically fastest approach.
As in Schoof's original algorithm~\cite{Sch1}, the basic strategy is to determine the trace of Frobenius $t_p$ modulo sufficient many small primes $\ell$.  By the Hasse bound, it suffices to do this for a set of primes whose product exceeds $4\sqrt{p}$.
The key improvement, due to Elkies, is a probabilistic method to determine $t$ modulo $\ell$ in $\ell(\ell+\log p)^{2+o(1)}$ expected time (see Theorem~\ref{thm:ElkiesPrime} for a more precise bound), provided that $\ell$ is an Elkies prime and $E$ is an ordinary elliptic curve with $j(E)\not\in \{0,1728\}$.
The Atkin primes also play a role in the algorithm, but their impact is asymptotically negligible and not considered here.
See~\cite{Sch2} for further details.

The standard heuristic complexity analysis of the SEA algorithm assumes there are approximately the same number of  
Atkin and Elkies primes $\ell < L$, where $L\sim \log p$, as $p\to\infty$;
see~\cite[\S17.2.2 and \S17.2.5]{ACDFLNV}, for example.
The validity or failure of this assumption  crucially affects  the expected running time of the SEA algorithm.
When it holds, the expected running time is $(\log p)^{4+o(1)}$ (see Corollary~\ref{cor:Elkies}).
It is known that the heuristic assumption regarding an approximately equal proportion of Atkin and Elkies primes $\ell < L$ starting with $L\sim \log p$, is not always true~\cite{Shp}; in some cases one may require a larger value of~$L$ (but this does not necessarily contradict the heuristic $(\log p)^{4+o(1)}$ bound on the 
expected running time  of the SEA algorithm). 
 
Little can be said about the worst-case running time of the SEA algorithm unconditionally, but under the Generalised Riemann Hypothesis (GRH) it can be bounded by $(\log p)^{8+o(1)}$ 
(see Corollary~\ref{cor:Elkies2}).
This follows from a result of Galbraith and Satoh~\cite[Appendix~A]{Sat}, who prove a GRH-based bound of $(\log p)^{2+o(1)}$ on the largest Elkies prime needed.\footnote{We note that~\cite[Appendix~A]{Sat} gives an expected time of $(\log p)^{3\mu+2+o(1)}$ for SEA under GRH, where $\mu$ is the exponent in multiplication, but, as confirmed to us by the authors, this bound is incorrect.  See Remark 2 in Section~\ref{sec:aux} for details.}

By comparison, the complexity of Schoof's original deterministic algorithm \cite{Sch1,Sch2} is just $(\log p)^{5+o(1)}$ (see Corollary~\ref{cor:Schoof} for a more precise bound).
Thus, even assuming the GRH, one can not prove that the SEA algorithm is actually an improvement over Schoof's algorithm, although in practice its performance is empirically superior.
There is therefore an interest in what can be said about the distribution of Elkies and Atkin primes ``on average".
In~\cite{ShpSuth} it is shown that for any sufficiently large prime $p$ almost all elliptic curves over $\F_p$ have, up to a constant factor, approximately the same number of Elkies and Atkin primes (unconditionally).
Here we consider the analogous question for the reductions $E_p$ of our fixed elliptic curve $E/\Q$ and obtain a similar result,
conditional on the GRH.

Traditionally, Elkies  and  Atkin primes $\ell$ are defined only for ordinary primes $p$.
For the purpose of stating (and proving) our results, it is convenient to extend the definition to all primes $p$; we address the ordinary/supersingular distinction when we discuss algorithmic applications. 

Thus for a prime $p>3$ of good reduction for $E$ and a real $L$, we 
define $R_a(p;L)$ and $R_e(p;L)$ as the number of Atkin and Elkies primes, respectively, in the dyadic interval $[L, 2L]$, for the elliptic curve $E_p$.
We clearly have
\begin{equation}
\label{eq:Sum RaRe}
R_a(p; L) + R_e(p; L) =  \pi(2L)-\pi(L) + O\(1\),
\end{equation}
where $\pi(z)$ denotes the number of primes $\ell < z$, and it is 
natural to expect that
\begin{equation}
\label{eq:NaNe}
R_a(p; L) \sim R_e(p; L)  \sim    \frac{\pi(2L)-\pi(L)}{2},
\end{equation}
as $L\to \infty$. 

Here we prove, under the GRH,
that for all sufficiently large $P$
the asymptotic relations in~\eqref{eq:NaNe} hold for almost 
all primes $p\in [P, 2P]$, for a wide range 
of parameters $L$ and $P$. Our analysis relies on 
a bound of sums of Jacobi symbols involving Frobenius discriminants $D_p$,
due to Cojocaru and David~\cite{CojDav}.

Throughout the paper all implied constants may depend on the fixed
elliptic curve~$E$. The letters $\ell$ and $p$, with and without subscripts, 
always denote prime  numbers.
Our main result is the following: 

\begin{theorem} 
\label{thm:AvElkies} Under the GRH, for $\nu =1,2$ and any real $L,P\ge 1$ we have
\begin{equation*}
\begin{split}
 \frac{1}{\pi(2P)-\pi(P)}  \sum_{p\in [P, 2P]}&
  \left|R_*(p;L) -  \frac{\pi(2L)-\pi(L)}{2}\right|^{2\nu}\\
 & =O\(  \frac{L^{\nu}}{(\log L)^{\nu}}
 + \frac{L^{8\nu}(\log P)^2}{ P^{1/2} (\log L)^{2\nu}} \),
  \end{split}
\end{equation*}
 where $R_*(p;L)$ is either $R_a(p; L)$ or $R_e(p; L)$. 
\end{theorem}

\begin{cor} 
\label{cor:AvElkies} Under the GRH, for $\nu =1,2$ and any real $L,P\ge 1$ 
there are at most $O\(PL^{-\nu} (\log L)^{\nu}(\log P)^{-1} + L^{6\nu} P^{1/2} \log P\)$ primes $p\in [P,2P]$ for which 
$$
R_*(p;L) < \frac13 (\pi(2L)-\pi(L)) ,
$$
 where  $R_*(p;L)$ is either $R_a(p; L)$ or $R_e(p; L)$. 
\end{cor}



It is easy to see  that Theorem~\ref{thm:AvElkies} and Corollary~\ref{cor:AvElkies}
give nontrivial bounds when
$$
\psi(P) 
\le L \le P^{1/12} \(\log P\)^{-1/3}\psi(P)^{-1}, 
$$
for any function $\psi(z) \to \infty$ as $z \to \infty$  
and all sufficiently 
large $P$. This comfortably includes the range of $L$ of 
order $\log P$ needed to guarantee
$$
\prod_{\substack{\ell \in [L,2L]\\\ell~{Elkies~prime}}} \ell > 4p^{1/2},
$$
which is relevant to the SEA algorithm, see~\cite[Theorem~13]{Suth2}.

As we have mentioned, the SEA algorithm does not apply to supersingular primes $p$.
However, such primes can be identified in $(\log p)^{3+o(1)}$ expected time~\cite[Proposition~4]{Suth1},
and by~\cite{Elk1}, there are only $O(P^{3/4})$ supersingular primes in $[P,2P]$.
Thus this does not affect our 
algorithmic applications.
We now apply Corollary~\ref{cor:AvElkies} with $\nu =2$ and
$L = 2\log P$.

\begin{cor} 
\label{cor:SEA} Under the GRH, for any real $P\ge 3$ 
 the SEA algorithm computes $N_p$ in $(\log p)^{4+o(1)}$ expected time for all
but
$$
O\(P(\log P)^{-2}  (\log \log P)^{2}\)
$$
primes $p \in [P,2P]$.
\end{cor} 

As noted above, Schoof's algorithm computes $N_p$ in time $(\log p)^{5+o(1)}$ for every prime $p$. Thus for any prime $p\in [P,2P]$, if we find that the SEA algorithm appears to be be taking significantly longer than the expected $(\log p)^{4+o(1)}$ time bound, we can revert to Schoof's algorithm (here we note that all our implied constants essentially come from the work of Lagarias and Odlyzko~\cite{LaOd} and can be made effective for the purposes of making this determination). This can happen for only an $O((\log \log P)^2/(\log P))$ proportion of the primes $p\in [P,2P]$, which means that the average time spent per prime $p\in [P,2P]$ is still $(\log p)^{4+o(1)}$.
Applying this approach to each subinterval in a dyadic partitioning of $[1,P]$, we obtain the following result.


\begin{theorem}\label{thm:all}
Let $E$ be an elliptic curve over~$\Q$ and let $P\ge 3$ be a real number.
Under the GRH there is a probabilistic algorithm to compute $N_p$ for  primes $p \le P$ of good reduction for $E$ in 
$(\log P)^{4+o(1)}$ average time using $(\log P)^{2+o(1)}$ average space. 
\end{theorem}

It is natural to compare Theorem~\ref{thm:all} to the recent remarkable result of 
Harvey~\cite{Harv1} that gives a deterministic algorithm to compute the number of points $N_p$ on the reductions $C_p$ of a fixed hyperelliptic curve $C/\Q$; see~\cite{Harv2,HarvSuth,HarvSuth2} for further developments and improvements.
Applying Harvey's result~\cite{Harv1} in genus 1 yields a deterministic algorithm with an unconditional amortised time complexity that
matches that of Theorem~\ref{thm:all}.
However, this amortised result is weaker than Theorem~\ref{thm:all}, since it assumes one is computing $N_p$ for all suitable primes $p\le P$, whereas Theorem~\ref{thm:all} applies to a randomly chosen $p\le P$.
Additionally, the space complexity
of the algorithm of~\cite{Harv1} is exponential in $\log P$ (even excluding the output), whereas the space complexity given by 
Theorem~\ref{thm:all} is polynomial in $\log P$
(even when computing $N_p$ for all suitable primes $p\le P$).

\section{Sums of Jacobi Symbols with Frobenius Discriminants}
\label{sec:CharSum}

We recall that
the notations $U \ll V$ and  $V \gg U$, which are both 
equivalent to the statement  $U = O(V)$. Throughout the paper
the implied constant may depend on the fixed elliptic curve $E$ and on the integer parameter $\nu \ge 1$. 
As usual, we use $(\frac{k}{m})$ to denote the Jacobi symbol
of integer $k$ modulo an odd integer $m \ge 3$.

We need the bound on sums of  Jacobi Symbols with Frobenius discriminants
given in~\cite[Theorem~3]{CojDav}, and also some of its modifications
modulo a product of four primes $\ell_1, \ell_2,\ell_3, \ell_4$.
For $m=\ell_1\ell_2$ (or $m=\ell_1\ell_2\ell_3\ell_4$ in our modified version),
these statements require
 the surjectivity of the mod-$m$ Galois representation
$$
\rho_{E,m}\colon\Gal(\Qbar/\Q)\to\Aut(E[m])\simeq \GL_2(\Z/m\Z)
$$
induced by the action of the absolute Galois group $\Gal(\Qbar/\Q)$ on the $m$-torsion subgroup $E[m]$ of $E(\Qbar)$.

By Serre's open image theorem~\cite{Ser}, when $E$ does not have complex multiplication the image of the adelic Galois representation
$$
\rho_E\colon \Gal(\Qbar/\Q)\to \Aut(E[\hat\Z])\simeq \GL_2(\hat \Z)
$$
has finite index $i_E$ in $\GL_2(\hat\Z)$ (as usual, $E[\hat\Z]$ denotes $\varprojlim E[m]$ and $\hat \Z$ denotes $ \varprojlim \Z/m\Z$).
There is thus a minimal integer $m_E$ for which the index of $\bar\rho_{E,m_E}$ in $\GL_2(\Z/m_E\Z)$ is equal to $i_E$, and for all integers $m$ coprime to $m_E$ (in particular, all $m$ whose prime divisors are sufficiently large), the representation $\bar\rho_{E,m}$ must be surjective.

With this understanding we now state~\cite[Theorem~3]{CojDav} in the form we need here.

\begin{lemma} 
\label{lem:CharSum2} Under the GRH, 
for all sufficiently large  $P$ and all sufficiently 
large distinct primes $\ell_1, \ell_2< P$, we have
\begin{equation*}
\begin{split}
\sum_{p \in [P,2P]} \(\frac{D_p}{\ell_1\ell_2}\)
=\bigl(\pi(2P) -\pi(P)&\bigr)\\
\prod_{i=1}^2 \(\frac{-1}{\ell_i}\) &\frac{1}{\ell_i^2-1} + 
O\(\ell_1^3 \ell_2^3 P^{1/2} \log P\).
  \end{split}
\end{equation*}
\end{lemma} 

We also need a straightforward generalisation of 
Lemma~\ref{lem:CharSum2} for  products of four primes.

\begin{lemma} 
\label{lem:CharSum4} Under the GRH, 
for all sufficiently large  $P$ and all sufficiently 
large distinct primes $\ell_1, \ell_2,\ell_3, \ell_4< P$,
we have
\begin{equation*}
\begin{split}
\sum_{p \in [P,2P]} \(\frac{D_p}{\ell_1\ell_2\ell_3  \ell_4}\)
= \bigl(\pi(2P) -\pi(P)&\bigr)\\
 \prod_{i=1}^4   \(\frac{-1}{\ell_i}\) & \frac{1}{\ell_i^2-1} + 
O\(\ell_1^3 \ell_2^3 \ell_3^3  \ell_4^3P^{1/2} \log P\).
  \end{split}
\end{equation*}
\end{lemma} 

\begin{proof} The proof proceeds identically to that of~\cite[Theorem~3]{CojDav}.
In particular we define
$$
\cC_\ell(1)= \frac{\ell^3-\ell^2}{2} - 
\left\{\begin{array}{ll}
 \ell &\quad\text{if} \  \ell \equiv 1 \pmod 4, \\
0&\quad\text{if} \  \ell \equiv 3 \pmod 4,
\end{array}\right.
$$
and 
$$
\cC_\ell(-1)= \frac{\ell^3-\ell^2}{2} - 
\left\{\begin{array}{ll}
 0 &\quad\text{if} \  \ell \equiv 1 \pmod 4, \\
\ell &\quad\text{if} \  \ell \equiv 3 \pmod 4. 
\end{array}\right.
$$

For $\xi = \pm 1$, let $\Gamma_\xi$ be be set of vectors 
$(\gamma_1, \gamma_2, \gamma_3, \gamma_4)$
with 
$$\gamma_1, \gamma_2, \gamma_3, \gamma_4 = \pm 1 \mand 
\gamma_1 \gamma_2  \gamma_3 \gamma_4 = \xi.
$$
We then set 
$$
A_\xi(\ell_1, \ell_2,\ell_3, \ell_4) = \sum_{(\gamma_1, \gamma_2, \gamma_3, \gamma_4)\in \Gamma_\xi}
\cC_{\ell_1}(\gamma_1)\cC_{\ell_2}(\gamma_2)
\cC_{\ell_3}(\gamma_3)\cC_{\ell_4}(\gamma_4), \quad \xi = \pm 1.
$$
Arguing as in~\cite{CojDav}, see, 
for example~\cite[Equation~(18)]{CojDav} and~\cite[Theorem~9]{CojDav}, 
we obtain
\begin{equation*}
\begin{split}
\sum_{p \in [P,2P]} &\(\frac{D_p}{\ell_1\ell_2\ell_3  \ell_4}\)\\
&=\bigl(\pi(2P) -\pi(P)\bigr) \frac{A_1(\ell_1, \ell_2,\ell_3, \ell_4) - A_{-1}(\ell_1, \ell_2,\ell_3, \ell_4) }
{(\ell_1^3 - \ell_1)(\ell_2^3 - \ell_2)(\ell_3^3 -\ell_3)(\ell_4^3-\ell_4)}\\
& \qquad\qquad\qquad\qquad\qquad\qquad\qquad\quad
 + O\(\ell_1^3 \ell_2^3 \ell_3^3  \ell_4^3P^{1/2} \log P\).
  \end{split}
\end{equation*}
A direct calculation shows that 
\begin{equation}
\label{eq:A-A}
A_1(\ell_1, \ell_2,\ell_3, \ell_4) - A_{-1}(\ell_1, \ell_2,\ell_3, \ell_4) = 
\prod_{i=1}^4 \(\frac{-1}{\ell_i}\) \ell_i
\end{equation}
holds for all odd primes $\ell_1,\ell_2,\ell_3,\ell_4$.
\end{proof}

\section{Prime Divisors of  Frobenius Discriminants}
\label{sec:PrimDiv}

To apply Lemmas~\ref{lem:CharSum2} and~\ref{lem:CharSum4} 
we also need to estimate the 
average number of prime divisors $\ell \in [L, 2L]$
of the  Frobenius discriminants~$D_p$. Our main tool is provided 
by David and Wu~\cite[Theorem~3.2]{DW1}; 
see also~\cite[Theorem~2.3]{DW2}
for a similar statement concerning $N_p$. 

As usual, we use $\phi(r)$ to denote the Euler function 
of an integer $r\ge2$. A combination of some of the ideas in~\cite[Theorem~3.2]{DW1}
and~\cite[Lemma~2.2]{DW2} yields the following estimate.

\begin{lemma} 
\label{lem:Cong} Under the GRH, 
for an odd square-free integer $r\ge 2$ and  sufficiently large  $P$, we have
$$
\#\{p \in [P,2P]~:~D_p \equiv 0 \pmod r\} \ll
\frac{P}{\phi(r) \log P} + 
r^3 P^{1/2} \log P .
$$
\end{lemma}

\begin{proof} The result follows from  an effective version of the Chebotarev Density Theorem exactly as~\cite[Theorem~3.2]{DW1}; accordingly, we refer to some notation 
of~\cite{DW1}. To derive the desired result,  we define the 
set of conjugacy classes $C(r)$ as follows:
$$
C(r) = \{ g \in \GL(2,\Z/r\Z)~:~ 4\det(g) \equiv \tr(g)^2 \pmod r\}
$$
(in one place in~\cite{DW1} the corresponding congruence is  $\det(g) +1\equiv \tr(g)  \pmod r$
which is inconsequential). Finally, we also use a full analogue 
of~\cite[Lemma~2.2]{DW2}, to 
upper bound the main term in the corresponding 
asymptotic formula. In fact, since $r$ is square-free, 
we only need the  part of~\cite[Lemma~2.2]{DW2} that relies on the 
Chinese Remainder Theorem, which generalises in a straightforward
fashion to the new congruence condition.
\end{proof}

We note that one can probably drop the condition that $r$ is square-free
in Lemma~\ref{lem:Cong}; however, this requires one to verify that the somewhat tedious 
lifting argument also works with the new congruence condition.

For an integer $d$ we denote by  $\omega_L(d)$  the number of primes  
$\ell\in[L,2L]$ for which $\ell \mid d$ (note that $\omega_L(0)=\pi(2L)-\pi(L)$ is well defined).

\begin{lemma} 
\label{lem:DivAver} Under the GRH, for any fixed integer $\nu=1, 2, \ldots$, 
and  sufficiently large  $P$, we have
$$
\sum_{p \in [P,2P]}\omega_L\(D_p\)^\nu \ll
\frac{P}{ \log L \log P} + \frac{L^{4\nu} P^{1/2} \log P}{(\log L)^{\nu}}.
$$
\end{lemma}

\begin{proof} We write
$$
\sum_{p \in [P,2P]}\omega_L\(D_p\)^\nu 
= \sum_{\substack{\ell_1, \ldots, \ell_\nu \in [L,2L]\\
\ell_1, \ldots, \ell_\nu~\text{prime}}}
\sum_{\substack{p \in [P,2P]\\ 
\lcm[\ell_1, \ldots, \ell_\nu] \mid D_p}} 1. 
$$
Collecting for each $j=1, \ldots, \nu$ 
the $O\(L^j(\log L)^{-j}\)$ terms with exactly~$j$
distinct primes $\ell_1, \ldots, \ell_\nu$
and noticing that in this case 
$$
\sum_{\substack{p \in [P,2P]\\ 
\lcm[\ell_1, \ldots, \ell_\nu] \mid D_p}} 1
\ll
\frac{P}{L^j \log P} + 
L^{3j} P^{1/2} \log P , 
$$
by Lemma~\ref{lem:Cong}, we
obtain
$$
\sum_{p \in [P,2P]}\omega_L\(D_p\)^\nu 
\ll \sum_{j=1}^\nu \frac{L^j}{(\log L)^{j}}
\(\frac{P}{L^j \log P} + 
L^{3j} P^{1/2} \log P\),
$$
and the result follows.  
\end{proof}

\section{Proof of Theorem~\ref{thm:AvElkies}}
Recall that $R_a(p;L)$ and $R_e(p;L)$ denote the number of Atkin and Elkies primes, respectively, in the dyadic interval $[L,2L]$, for the elliptic curve $E_p$ (the reduction of our fixed elliptic curve $E/\Q$ modulo $p$).

We clearly have 
$$
R_a(p; L) - R_e(p; L)  = \sum_{\ell \in [L,2L]} \(\frac{D_p}{\ell}\)
+ O\(\omega_L(D_p)\),
$$
where, as before, $\omega_L(d)$   denotes the number of primes  
$\ell\in[L,2L]$ for which $\ell \mid d$.

Therefore,  by the H{\"o}lder inequality,
\begin{equation}
\label{eq:UV}
 \sum_{p\in [P, 2P]} \left|R_a(p; L) - R_e(p; L) \right|^\nu 
\ll  U+    V  + 1, 
\end{equation}
where
$$
U = 
\sum_{p\in [P, 2P]} 
\left| \sum_{\ell \in [L,2L]} \(\frac{D_p}{\ell}\)\right|^{2\nu}
\mand 
V = 
\sum_{p\in [P, 2P]}  \omega_L(D_p)^{2\nu}.
$$

We now consider the case of $\nu = 2$.
In this case, changing the order of summation, 
we obtain
$$
 U \le 
\sum_{\ell_1,\ell_2, \ell_3,\ell_4 \in [L,2L]}
 \sum_{p \in [P,2P]} \(\frac{D_p}{\ell_1\ell_2\ell_3  \ell_4}\).
$$
Without loss of generality we can assume that $2L < P$; otherwise the 
bound is trivial.

We estimate the sum over $p$ differently depending on the 
number of repeated values among $\ell_1,\ell_2, \ell_3,\ell_4$. 

\begin{itemize}
\item For $O(L^2/(\log L)^2)$ choices of $(\ell_1,\ell_2, \ell_3,\ell_4)$ 
for which the product $\ell_1\ell_2\ell_3  \ell_4$ is a perfect square,
we estimate the inner sum trivially as $O(P/\log P)$.

\item  For 
$O\(L^3/(\log L)^3\)$ choices of $(\ell_1,\ell_2, \ell_3,\ell_4)$ 
for which $\ell_1\ell_2\ell_3  \ell_4$ is not a perfect square
but is divisible by a nontrivial square, we use Lemma~\ref{lem:CharSum2}. 

\item For the remaining  choices of $(\ell_1,\ell_2, \ell_3,\ell_4)$ 
we use  Lemma~\ref{lem:CharSum4}.
\end{itemize}

Therefore, we find that
\begin{equation*}
\begin{split}
U
& \ll \frac{L^2}{(\log L)^2}\cdot\frac{P}{  \log P}
 + \frac{L^3}{(\log L)^3}\(\frac{P}{L^4\log P}+ L^{6} P^{1/2}\log P\)\\
 & \qquad \qquad \qquad \qquad \quad + 
  \frac{L^4}{(\log L)^4} \( \frac{P}{L^8\log P} + L^{12} P^{1/2}\log P\),
   \end{split}
\end{equation*}
which after removing the terms that never dominate, yields the bound
\begin{equation}
\label{eq:U2}
 U \ll  \frac{L^2P }{(\log L)^2 P} + \frac{L^{16} P^{1/2} \log P}{(\log L)^{4}} .
\end{equation}
Furthermore, by Lemma~\ref{lem:DivAver} we
have
\begin{equation}
\label{eq:V}
V \ll \frac{P}{ \log L \log P} + \frac{L^{16} P^{1/2} \log P}{(\log L)^{4}}.
\end{equation}
Substituting~\eqref{eq:U2} and~\eqref{eq:V} in~\eqref{eq:UV}
and noticing that the estimate on $U$ always dominates 
that on $V$, 
we obtain
\begin{equation*}
\begin{split}
 \sum_{p\in [P, 2P]}& \left|R_a(p; L) - R_e(p; L) \right|^{4} \ll 
 \frac{L^2P }{(\log L)^2 P} + \frac{L^{16} P^{1/2} \log P}{(\log L)^{4}} 
  \end{split}.
\end{equation*}
Combining this with~\eqref{eq:Sum RaRe}, 
we  conclude the proof for $\nu=2$.

The case $\nu =1$ is completely analogous albeit technically easier,
since we only have to use Lemma~\ref{lem:CharSum2}. 

\section{Some Auxiliary Estimates}
\label{sec:aux}

Here we take the opportunity to clarify and record  stronger 
versions of several relevant complexity bounds that have previously appeared
in the literature in less precise forms (and in some cases, with errors).
In this section $E$ denotes an elliptic curve over a finite field $\F_p$, where $p >3$ is prime and $E$ defined by an equation of the form $Y^2=f_E(X)$, where $f_E\in\F_p[X]$ is a monic square-free cubic.

We assume throughout that algorithms based on the fast Fourier transform (FFT) are used for multiplication.  This allows us to bound the time to multiply two $n$-bit integers by
\begin{equation}\label{eq:Mn}
\M(n) = O(n\log n\llog n),
\end{equation}
via the result of Sch\"onhage and Strassen~\cite{SS}. We note that this bound can be improved slightly \cite{Fur,HvdHL}, but we do not use this improvement.

The bound in \eqref{eq:Mn} not only asymptotically valid, it is practically relevant.
Using Kronecker substitution~\cite[\S8.4]{GG}, one can reduce the problem of multiplying two polynomials in $\F_p[X]$ of degree at most~$d$ to the multiplication of two integers with approximately $2d\log_2(dp))$ bits.
Even when $\log_2 p$ is not particularly large, $2d\log_2(dp)$ may easily be large enough to justify the use of the FFT; this applies, in particular, to algorithms for computing $\#E(\F_p)$ over cryptographic size fields, where $2d\log_2(dp)$ may easily exceed $10^5$ or $10^6$, even though $\log_2 p < 10^3$.

We also note the following complexity bounds for arithmetic in $\F_p$ and $\F_p[X]$, which follow from standard fast algorithms for division with remainder (see \cite[Ch.~9]{GG}) and the extended Euclidean algorithm (see \cite[Ch.~11]{GG}), combined with
Kronecker substitution.

\begin{lemma}\label{lem:arith}
Let $n = \rf{\log_2 p}$, let $a,b\in\F_p^\times$, let $f,g\in\F_p[X]$ be nonzero polynomials of degree at most $d$, and assume $\log d=O(n)$.
The following bounds hold:
\begin{center}
\setlength{\tabcolsep}{20pt}
\begin{tabular}{ll}
operation & complexity\\
\midrule
$ab$ & $O(\M(n))$\\
$a^{-1}$ & $O(\M(n)\log n)$\\
$fg$ & $O(\M(dn))$\\
$f \bmod g$ & $O(\M(dn))$\\
$\gcd(f,g)$ &  $O(\M(dn)\log d)$\\
\bottomrule
\end{tabular}
\end{center}
\vspace{8pt}
When $\gcd(f,g)=1$, the multiplicative inverse of the reduction of $f$ in the ring $\F_p[X]/(g)$ can be computed in time $O(\M(dn)\log d)$.
\end{lemma}

\subsection{Schoof's algorithm}

Let $\pi$ denote the Frobenius endomorphism of $E/\F_p$.
Schoof's algorithm computes $\#E(\F_p)$ by computing $t=\tr\ \pi$ modulo a set of primes $\ell$ whose product exceeds $4\sqrt{p}$, and then uses the Chinese Remainder Theorem to determine $t$.
By the Prime Number Theorem (PNT), $O(\log p)$ primes suffice.
To simplify matters, we restrict our attention to odd primes $\ell\ne p$.

The Frobenius endomorphism $\pi$ induces an endomorphism $\pi_\ell$ of $E[\ell]$ that satisfies the characteristic equation
$$
\pi_\ell^2-t_\ell\pi_\ell + p_\ell = 0
$$
in the ring $\End(E[\ell]):=\End(E)/(\ell)$.
Here $t_\ell$ and $p_\ell$ denote the elements of $\End([\ell])$ induced by scalar multiplication by $t$ and $p$, respectively.
Schoof's algorithm works by explicitly computing $\pi_\ell^2+p_\ell$ and $\pi_\ell, 2\pi_\ell,3\pi_\ell,\ldots$, using addition in $\End(E[\ell])$, until it finds a multiple of~$\pi_\ell$ that is equal to $\pi_\ell^2+p_\ell$; this multiple determines $t\bmod \ell$.
In order to give precise complexity bounds, we now sketch an explicit implementation of the algorithm; the presentation here differs slightly from that given by Schoof in \cite{Sch1,Sch2}, but it yields sharper results.

Let $\psi_\ell(X)$ denote the $\ell$th division polynomial of $E$; it is a polynomial of degree $(\ell^2-1)/2$ whose roots are the $x$-coordinates of the nonzero points in the $\ell$-torsion subgroup $E[\ell]$.
One can recursively define polynomials $f_0,f_1,\ldots f_k\in \F_p[X]$, depending on the coefficients of $E$, such that for odd integers $k$ the polynomial $f_k$ is precisely the $k$th division polynomial $\psi_k$; see \cite[\S 4.4.5a]{ACDFLNV}, for example.
The polynomials $f_k$ satisfy recursion relations that allow one to compute any particular $f_k$ using a double-and-add approach.  Each step involves $O(1)$ multiplications of polynomials of degree $O(k^2)$, and since $k$ is roughly doubling with each step, the total cost is dominated by the last step.
This allows one to compute $\psi_\ell(X)$ in $O(\M(\ell^2 n))$ time using $O(\ell^2 n)$ space.

Nonzero elements of $\End(E[\ell])$ can be uniquely represented as ordered pairs of elements of the ring $R=\F_p[X,Y]/(\psi_\ell(X),Y^2-f_E(X))$, of the form $\varphi =(\alpha(X),\beta(X)Y)$.
The endomorphism $\varphi$ sends a nonzero point 
$(x_0,y_0)\in E[\ell]$ to the point $(\alpha(x_0),\beta(x_0)y_0)\in E[\ell]$.
Addition in the ring $\End(E[\ell])$ uses the algebraic formulas for the elliptic curve group law applied to ``points" of the form $(\alpha(X),\beta(X)Y)$.
The cost of addition is dominated by the cost of an inversion in $\F_p[X]/(\psi_\ell(X))$, which is $O(\M(\ell^2n)\log\ell)$.
By switching to projective coordinates, we can avoid inversions and reduce the complexity to $O(\M(\ell^2 n))$; testing the equality of two projectively represented elements of $\End(E[\ell])$ involves $O(1)$ multiplications in $\F_p[X]/(\psi_\ell(X))$ and has the same complexity.

The Frobenius endomorphism is represented by the ordered pair
$$
(X^p,Y^p)=(X^p,f_E(X)^{(p-1)/2}Y),
$$
which is computed by exponentiating the polynomials $X$ and $f(X)$ in the ring $\F_p[X]/(\psi_\ell(X))$. Using the standard square-and-multiply algorithm for fast exponentiation, this takes $O(\M(\ell^2 n)n)$ time, and the same applies to computing $\pi_\ell^2$.
The endomorphism $p_\ell$ is computed as a scalar multiple of the identity endomorphism $(x,y)$; using a double-and-add approach in projective coordinates, it takes $O(\M(\ell^2n)\log\ell)$ time to compute $p_\ell$.

\begin{theorem}\label{thm:SchoofTrace}
Let $\ell\ne p$ be an odd prime, and assume $\log \ell = O(n)$, where $n = \rf{\log_2 p}$.
With the implementation described above, given an elliptic curve $E/\F_p$, Schoof's algorithm computes the trace of Frobenius modulo $\ell$ in
$$
O(\M(\ell^2n)(\ell+n))
$$
time, using $O(\ell^2 n)$ space.
\end{theorem}
\begin{proof}
The time to compute $\psi_\ell(X)$ is $O(\M(\ell^2 n))$.
The time to compute $\pi_\ell$ and $\pi_\ell^2$ is $O(\M(\ell^2n)n)$.
The time to compute $p_\ell$ is $O(\M(\ell^2n)\log\ell)$, and this dominates the time to add $\pi_\ell^2$ and $p_\ell$.
Computing each multiple $m\pi_\ell$ by adding $\pi_\ell$ to $(m-1)\pi_\ell$ takes time $O(\M(\ell^2n)$,
as does comparing $m\pi_\ell$ and $\pi_\ell^2+p_\ell$.
We compute at most $\ell$ multiples of $\pi_\ell$ before finding a match, giving a total cost of $O(\M(\ell^2n)\ell)$ for the linear search.
Summing the bounds above yields a total time of $O(\M(\ell^2n)(\ell+n))$.
We store just $O(1)$ elements of the ring $\F_p[X]/(\psi_\ell(X))$ at any one time, so the space complexity is $O(\ell^2 n)$, including space for $\psi_\ell(X)$.
\end{proof}

\begin{corollary}\label{cor:Schoof}
With the implementation described above, Schoof's algorithm computes the Frobenius trace of an elliptic curve $E/\F_p$ in
$$
O(n^5\llog n)
$$
time, using $O(n^3)$ space, where $n = \rf{\log_2 p}$.
\end{corollary}

\begin{proof}
By the PNT, the primes $\ell$ used in Schoof's algorithm satisfy $\ell = O(n)$, and there are $O(n/\log n)$ of them.
The time for each $\ell$ is bounded by $O(\M(n^3)n)=O(n^4\log n\llog n)$.
Multiplying this by $O(n/\log n)$ gives the desired time bound, which dominates the time required to recover $t$ using the Chinese Remainder Theorem.
The space bound follows from the $O(\ell^2n)=O(n^3)$ space used per prime $\ell$ and the $O(n)$ spaced needed to store the value $t\bmod \ell$ for each $\ell$.
\end{proof}

\subsection{Identifying Elkies primes}
As above, let $\ell\ne p$ denote an odd prime.
We recall that
$$
E[\ell]\simeq \Z/\ell \Z\times \Z/\ell \Z,
$$
which we may regard as an $\F_\ell$-vector space.
After fixing a basis for $E[\ell]$, each nonzero endomorphism of $E$ determines a matrix in $\GL(2,\F_\ell)$ given by its action on the basis.
The characteristic polynomial of the matrix of the Frobenius endomorphism is precisely the characteristic polynomial of $\pi_\ell$, which does not depend on the choice of basis.

As observed by Elkies, if $t^2-4p$ is a quadratic residue modulo $\ell$ (meaning that $\ell$ is an Elkies prime), then the characteristic polynomial of $\pi_\ell$ splits into linear factors:
$$
X^2-t_\ell X+p_\ell = (X-\lambda_1)(X-\lambda_2)=0.
$$
Here $\lambda_1,\lambda_2\in \F_\ell^*$ are eigenvalues of the matrix of Frobenius in $\GL(2,\F_\ell)$, and it follows that the Frobenius endomorphism fixes at least one linear subspace of $E[\ell]$ (it may fix 1, 2, or $\ell+1$
distinct linear subspaces).
This subspace is an order-$\ell$ subgroup of $E[\ell]$ that is the kernel of a
separable isogeny $\varphi\colon E\to \widetilde E$ of degree $\ell$ (an $\ell$-\emph{isogeny}) that is defined over $\F_p$.

Conversely, if $E$ admits an $\F_p$-rational $\ell$-isogeny, this isogeny is separable, since $\ell\ne p$, and its kernel is an order-$\ell$ subgroup of $E[\ell]$ that is fixed by Frobenius; this implies that the characteristic polynomial of~$\pi_\ell$ splits and that $\ell$ is an Elkies prime.
Thus an odd prime $\ell\ne p$ is an Elkies prime if and only if $E$ admits an $\F_p$-rational $\ell$-isogeny.

We now recall the classical modular polynomial $\Phi_\ell\in\Z[X,Y]$ that parametrises pairs of $\ell$-isogenous elliptic curves in terms of their $j$-invariants. Note that in general, $\Phi_N$ parametrises $N$-isogenies with a cyclic kernel, but when $N=\ell$ is prime the kernel is necessarily cyclic.
The modular polynomial $\Phi_\ell$ has the defining property that over any field $\F$ of characteristic different from $\ell$, the modular equation
$$
\Phi_\ell(j_1,j_2) =0
$$
holds if and only if $j_1$ and $j_2$ are the $j$-invariants of elliptic curves $E_1/\F$ and $E_2/\F$ that are related by an $\F$-rational $\ell$-isogeny $\varphi\colon E_1\to E_2$.

Given an elliptic curve $E/\F_p$, to determine if $\ell$ is an Elkies prime for~$E$, it suffices to check whether the instantiated polynomial
$$
\varphi_\ell(X)=\Phi_\ell(j(E),X)\in \F_p[X]
$$
has a root in $\F_p$; any such root is necessarily the $j$-invariant of an $\ell$-isogenous elliptic curve defined over $\F_p$.

The polynomial $\Phi_\ell(X,Y)$ has degree $\ell+1$ in both $X$ and $Y$, and the size of its largest coefficient is $O(\ell \log \ell)$ bits (see~\cite{BS} for an explicit bound).
It can be computed using a probabilistic algorithm that, under the GRH, runs in $O(\ell^3(\log\ell)^3\llog\ell)$ expected time, using $O(\ell^3\log\ell)$ space \cite{BLS}.
Given $\Phi_\ell$, the time to compute $\varphi_\ell$ is $O(\ell^2\M(\ell\log\ell+n))$, 
where $n = \rf{\log_2 p}$.
Alternatively, there is a probabilistic algorithm to directly compute $\varphi_\ell$ that, under the GRH, runs in
$$
O(\ell^3(\log\ell)^3\llog \ell + \ell^2n(\log n)^2\llog n)
$$
expected time, using just $O(\ell n +\ell^2\log(\ell n))$ space \cite{Suth2}.
Having computed $\varphi_\ell$, we can determine whether it has any roots in $\F_p$ 
by computing $\gcd(X^p-X,\varphi_\ell(X))$.

We note that the probabilistic algorithms we consider here are all of \emph{Las Vegas} type, meaning that their output is always correct, it is only their running times that may depend on random choices.

\begin{theorem}\label{thm:ElkiesPrime}
Assume the GRH, and let $\ell\ne p$ be an odd prime with $\log\ell =O(n)$, where $n = \rf{\log_2 p}$.  The following hold.
\begin{enumerate}[\rm(a)]
\item There is a Las Vegas algorithm that  decides whether $\ell$ is an Elkies prime in
$O(\ell^3(\log \ell)^3\llog\ell + \ell n^2\log n\llog n)$ 
expected time,\\
using $O(\ell n + \ell^2\log(\ell n))$ space.
\item 
There is a  deterministic algorithm that decides whether $\ell$ is an Elkies prime in
$O(\ell^3(\log\ell)^2\llog\ell + \ell n^2\log n\llog n)$
time,\\
using $O(\ell^3\log\ell + \ell^2 n)$ space, assuming $\Phi_\ell$ is given.
\end{enumerate}
\end{theorem}

\begin{proof}
With fast exponentiation it takes $O(\M(\ell n)n)$ time to compute $X^p\bmod \varphi_\ell(X)$, dominating the time to compute $\gcd(X^p-X,\varphi_{\ell,E}(X))$,
by Lemma~\ref{lem:arith}.
If $n \le \ell$ this is bounded by $O(\ell^3\log\ell\llog\ell)$, which is dominated by the first term in both time bounds.
If $n > \ell$ this is bounded by $O(\ell n^2\log n\llog n)$, which is included in both time bounds.
The first time bound dominates the time to compute $\varphi_\ell$, and the second time bound dominates the time to compute $\varphi_\ell$ given $\Phi_\ell$ (consider the cases $n\le \ell\log\ell$ and $n > \ell\log\ell)$).
The space bounds follow immediately from the discussion above.
Finally, note that if $\Phi_\ell$ is given, computing $\varphi_\ell(X)=\Phi_\ell(j(E),X)$ and 
$\gcd(X^p-X,\varphi_\ell(X))$ does not involve the use of any probabilistic algorithms.
\end{proof}

\subsection{Elkies' algorithm}

We now consider the complexity of computing the Frobenius trace $t$ of $E/\F_p$ modulo an Elkies prime $\ell$.
Elkies' algorithm is similar to Schoof's algorithm, except rather than working modulo the $\ell$th division polynomial $\psi_\ell(X)$, it works modulo a \emph{kernel polynomial} $h_\ell(X)$ whose roots are the $x$-coordinates of the nonzero points in the kernel of an $\F_p$-rational $\ell$-isogeny $\varphi\colon E\to \widetilde E$.
The kernel polynomial $h_\ell$ necessarily divides the division polynomial $\psi_\ell$, since $\ker\varphi$ is a subgroup of $E[\ell]$, and it has degree $(\ell-1)/2$, rather than $(\ell^2-1)/2$, which speeds up the algorithm by a factor of at least~$\ell$.

Elkies assumes in~\cite{Elk2} that $E$ is not supersingular, and that $j(E)$ is not~0 or 1728; these restrictions are not a problem, since in any of these special cases there are alternative methods to compute $t$ that are faster than Elkies' algorithm.

Elkies~\cite{Elk2} gives an algorithm to compute the kernel polynomial $h_\ell(X)$ using the instantiated modular polynomial $\varphi_\ell(X)=\Phi_\ell(j(E),X)$, along with various instantiated partial derivatives of $\Phi_\ell(X,Y)$ that can either be computed directly using the algorithm in \cite{Suth2} or derived from $\Phi_\ell$ and instantiated.
The first step is to find a root of $\varphi_\ell$ in $\F_p$, which is necessarily the $j$-invariant of an elliptic curve $\widetilde E$ that is the image of an $\ell$-isogeny $\varphi\colon E\to \widetilde E$.
Using Rabin's probabilistic algorithm~\cite{Rab}, this can be accomplished in $O(\M(\ell n)n)$ expected time, assuming $\log\ell = O(n)$.
Once this has been done, one computes $h_\ell$ using \cite[Alg.\ 27]{Gal}, which takes $O(\ell^2\M(n)+\ell\M(n)\log n)$ time.

\begin{theorem}\label{thm:Elkies}
Assume the GRH, and let $\ell\ne p$ be an odd prime with $\log \ell = O(\log p)$.
Let $E/\F_p$ be an ordinary elliptic curve with $j(E)\not\in\{0,1728\}$.
If $\ell$ is an Elkies prime for $E$, then one can compute the Frobenius trace $t$ modulo $\ell$ in
\begin{enumerate}[\rm (a)]
\item $O(\ell^3(\log\ell)^3\log\log\ell + \ell n^2\log n\llog n)$
expected time,\\
using $O(\ell n + \ell^2\log(\ell n))$ space;
\item $O(\ell^3(\log\ell)^2\llog\ell + \ell n^2\log n\llog n)$
expected time,\\
using $O(\ell^3\log\ell + \ell^2 n)$ space,  if $\Phi_\ell$ is given.
\end{enumerate}
\end{theorem}

\begin{proof}
Theorem~\ref{thm:ElkiesPrime} bounds the complexity of computing $\varphi_\ell$ and determining whether it has a root in $\F_p$, both when $\Phi_\ell$ is given and when it is not.
In both cases, these bounds dominate the complexity of finding a root of $\varphi_\ell$ computing the kernel polynomial $h_\ell$.
Once $h_\ell$ has been computed, $t \bmod \ell$  can be computed in $O(\M(\ell n)(\ell+n))$ time using $O(\ell n)$ space; the argument is the same as in Theorem \ref{thm:SchoofTrace}, except the degree of $h_\ell$ is $O(\ell)$ rather than $O(\ell^2)$.
These bounds are dominated by both sets of bounds above.
\end{proof}

The bounds in Theorem~\ref{thm:Elkies} are the same as the corresponding bounds in Theorem~\ref{thm:ElkiesPrime}; the complexity of determining if $\ell$ is an Elkies prime dominates the complexity of computing $t$ modulo an Elkies prime.

\begin{corollary}\label{cor:Elkies}
Let $E/\F_p$ be an elliptic curve, and suppose that the least integer $L$ for which the product of the Elkies primes $\ell\le L$ exceeds $4\sqrt{p}$ is $O(\log p)$.  Let $n=\rf{\log_2 p}$.
There is a Las Vegas algorithm to compute the Frobenius trace $t$ of $E$ in 
\begin{enumerate}[\rm(a)]
\item $O(n^4(\log n)^2\log\log n)$
expected time, using $O(n^2\log n)$ space;
\item $O(n^4\log n\log\log n)$
expected time, using $O(n^4)$ space, if the modular polynomials $\Phi_\ell$
for all primes $\ell\le L$ are precomputed. 
\end{enumerate}
\end{corollary}

\begin{proof}
We first determine whether $E$ is supersingular or not; using the algorithm in \cite{Suth1} this can be done in $O(n^3\log n\llog n)$ expected time using $O(n)$ space.
If $E$ is supersingular then $t\equiv 0\bmod p$, and for $p\ge 5$ the Hasse bound $|t|\le 2\sqrt{p}$ implies $t=0$
(for $p\le 3$ we can count points na\"ively and output $p+1-\#E(\F_p)$).

If $j(E)=0$ then $E$ has CM by $\Q(\sqrt{-3})$, and the norm equation $4p=t^2+3v^2$ can be solved using Cornacchia's algorithm in $O(n^2)$ time.
This determines at most 6 possibilities for $t$; the correct one can be distinguished using \cite[Alg.\ 3.5]{RS}.
Similarly, if $j(E)=1728$ then $E$ has CM by $\Q(i)$, so we solve $4p=t^2+v^2$ and apply \cite[Alg.\ 3.4]{RS}.

Otherwise, we apply Theorem~\ref{thm:Elkies} to each Elkies prime $\ell\le L$.
There are $O(n/\log n)$ such primes, each bounded by $O(n)$.
This yields the desired complexity bounds, which dominate the complexity of recovering~$t$ using the Chinese Remainder Theorem.
\end{proof}

\begin{remark}
The $O((\log p)^2)$ space complexity bound for SEA listed in~\cite[p.~421]{ACDFLNV} is incorrect; 
the space complexity of the algorithm given there is $\Omega((\log p)^3)$ (consider line~3 of \cite[Alg.~17.25]{ACDFLNV}, for example).
\end{remark}

\subsection{Bounding Elkies primes}

We now sharpen the bound of Galbraith and Satoh~\cite[Theorem~5]{Sat}
on the size of an interval in which one can guarantee the existence
sufficiently many Elkies primes, assuming the GRH.

We recall the classical bound, see~\cite[Ch.~13]{Mont}, 
that asserts that under the GRH, for any integer $D\ge 2$, 
\begin{equation}
\label{eq:GRH Char1}
\sum_{n \le L} \(1 - \frac{n}{L}\) \(\frac{D}{n}\) \Lambda(n)= O(L^{1/2} \log D),
\end{equation}
where  
$\Lambda(n)$ denotes the von~Mangoldt function given by
$$
\Lambda(n)=
\begin{cases}
\log \ell &\quad\text{if $n$ is a power of the prime $\ell$,} \\
0&\quad\text{if $n$ is not a prime power.}
\end{cases}
$$
After discarding the contribution $O(L^{1/2})$ from $O(L^{1/2}/\log L)$ 
prime powers up to $L$, we see that~\eqref{eq:GRH Char1} is equivalent
to 
\begin{equation}
\label{eq:GRH Char}
\sum_{\ell\le L} \(1 - \frac{\ell}{L}\) \(\frac{D}{\ell}\) \log \ell= O(L^{1/2} \log D).
\end{equation}

Let  $R$ and $R_0$ be the number of primes  $\ell \le L$ such $D$ is a
quadratic residue modulo $\ell$ and such that  $\ell \mid D$, respectively.
Let $M$ is the smallest integer with $\pi(M) = \pi(L) - R - R_0$.
Therefore, by the PNT and partial summation,
\begin{equation}
\label{eq:Triv}
\begin{split}
\sum_{\ell \le L} \(1 - \frac{\ell}{L}\) \(\frac{D}{\ell}\) \log \ell & \le  -  \sum_{\ell \le M} \(1 - \frac{\ell}{L}\)  \log \ell +  R\log L   \\
  & =  -  \(1 - \frac{M }{2L}+o(1)\)M   + R\log L   \\
  & \le  -  \(\frac{1 }{2}+o(1)\) M + R\log L  .
  \end{split}
\end{equation}
Since $R_0 = O(\log D)$, we see that if $L\gg (\log D)^2$
then $R_0 = o(L)$. If $R> L/(5\log L)$ there is nothing to prove.
Otherwise, applying the PNT again, we obtain 
$$
 M\ge \(\frac{4}{5}+o(1)\) L \mand R\log L \le  \(\frac{1 }{5}+o(1)\) L.
 $$
Substituting these bounds in~\eqref{eq:Triv}, we derive
\begin{equation}
\label{eq:Conts}
\sum_{\ell \le L} \(1 - \frac{\ell}{L}\) \(\frac{D}{\ell}\) \log \ell
\le -  \(\frac{1 }{5}+o(1)\)L.
\end{equation}
Now, recalling~\eqref{eq:GRH Char}   and 
taking $L \ge C(\log D)^2$ we see that~\eqref{eq:Conts}
is impossible and thus in this case  $ R \ge L/(5\log L)$. 
Note that using the estimates of~\cite{LaLiSo} one can get a completely 
explicit version of this estimate, with explicit constants. 
In particular, this means that one can simply take $C (\log p)^2$ 
in~\cite[Theorem~5]{Sat}.
Thus for an 
appropriate absolute constant $C> 0$, for any $L \ge C(\log D)^2$
there are  at least $L/(5\log L)$ Elkies primes up to $L$. 
In the SEA algorithm we can simply take $L =  C(\log D)^2$.

\begin{corollary}\label{cor:Elkies2}
Under the GRH, the expected running time of the SEA algorithm is
$O(n^8(\log n)^2\llog n)$.
\end{corollary}

\begin{remark}
If one assumes that the reduced polynomials $\Phi_\ell\bmod p$ have been precomputed for all $\ell\le L$, the bound in Corollary~\ref{cor:Elkies2} can be improved to $O(n^7\llog n)$; this assumption does not make sense in our setting, where $p$ is varying, but it might be appropriate if many computations use the same prime~$p$, as in \cite{Sat}.  As noted in the introduction, the bound $(\log p)^{3\mu+2+o(1)}$ given in \cite[Appendix A]{Sat} is incorrect; under the assumption that all $\Phi_\ell\bmod p$ are precomputed (as assumed there), the bound should be $(\log p)^{\max(\mu+6,3\mu+3)+o(1)}$, where $\mu\in[1,2]$ has the property that two $n$-bit integers can be multiplied in time $n^{\mu+o(1)}$ (so in fact one can take $\mu=1$).
\end{remark}

We should note that the bound in Corollary~\ref{cor:Elkies2} is of purely philosophical interest.
As a practical matter, there is no reason to ever apply Elkies' algorithm to 
primes $\ell \gg  n^{4/3}$, since for such $\ell$ one can use Schoof's algorithm to compute the Frobenius trace $t\in \Z$ more quickly than one can compute $t\bmod \ell$ using Elkies' algorithm.
More generally, one may adopt a hybrid approach as follows.
Enumerate odd primes $\ell\ne p$ in increasing order.
If $\ell$ is an Elkies prime, use Elkies' algorithm to compute $t\bmod \ell$, otherwise, add $\ell$ to a list $S$ that contains all previously considered primes $\ell$ for which $t\bmod \ell$ is not yet known.
Before determining whether the next prime $\ell$ is an Elkies prime, first check whether $\ell^{3/4} > c\ell_0$,
where $\ell_0=\min(S)$ and $c$ is a suitably chosen constant.
If this condition holds, then compute $t\bmod \ell_0$ using the method of Schoof, remove $\ell_0$ from $S$, and repeat.
Terminate as soon as the value of $t$ is known modulo a set of primes whose product exceeds $4\sqrt{p}$.
This approach guarantees an expected running time of $n^{5+o(1)}$, and heuristically achieves an expected running time of $n^{4+o(1)}$. 

%
%

\section{Comments}

In principle one can extend Lemmas~\ref{lem:CharSum2} and~\ref{lem:CharSum4}
to any number of primes $\ell_1, \ldots, \ell_{2\nu}$. 
However, one needs a general argument for computing
$A_1(\ell_1, \ldots, \ell_{2\nu}) - A_{-1}(\ell_1, \ldots, \ell_{2\nu})$,
analogous to that given in~\eqref{eq:A-A}. 
Using such  an extension one can consider larger values of $\nu$ in 
Theorem~\ref{thm:AvElkies} and Corollary~\ref{cor:AvElkies}.

It is shown in~\cite{Shp} that the bound of~\cite{ShpSuth},
which applies to almost all curves, cannot be extended 
to all curves modulo all primes. It would be interesting 
to try to derive a ``horizontal'' analogue of this lower 
bound.

We note that one can obtain an unconditional analogue
of Theorem~\ref{thm:AvElkies} as all the necessary tools (Lemmas~\ref{lem:CharSum2}, \ref{lem:CharSum4} and~\ref{lem:Cong}),
admit unconditional analogues; see~\cite{CojDav,DW1,DW2}.
However such a result requires $L$ to be smaller than $\log P$
which is not suitable for applications to the SEA algorithm.

\section*{Acknowledgements}

The authors thank Steven Galbraith and Takakazu Satoh for their help in 
clarifying the bounds in~\cite[Appendix A]{Sat}, and the referees for a
careful reading of the manuscript and several useful comments.

The inspiration for this work occurred during a very enjoyable stay of the authors
at the CIRM, Luminy.

During the preparation of this result, 
I.~E.~Shparlinski was supported in part
by ARC grant DP140100118 and
A.~V.~Sutherland received financial support from NSF grant DMS-1115455.


\begin{thebibliography}{9999}


\bibitem{ACDFLNV}
R. Avanzi, H.~Cohen, C.~Doche, G.~Frey, T.~Lange,  K.~Nguyen and
F. Vercauteren,
{\it Elliptic and hyperelliptic curve cryptography: Theory and practice\/},
CRC Press, 2005.



\bibitem{BS}
R. Br\"{o}ker and A. V. Sutherland, `An explicit height bound for the classical modular polynomial', {\it Ramanujan Journal} {\bf 22} (2010), 293--313.

\bibitem{BLS}
R. Br\"{o}ker, K. Lauter and A. V. Sutherland,
`Modular polynomials via isogeny volcanoes',
{\it Math. Comp.} {\bf 81} (2011), 1201--1231.

%

\bibitem{CojDav} A. C. Cojocaru and C. David,
`Frobenius fields for elliptic curves',
{\it Amer. J. Math.\/}, {\bf 130} (2008), 1535--1560.


\bibitem{DW1} C. David and J. Wu, 
`Almost-prime orders of elliptic curves over finite fields',
{\it Forum Math.\/},  {\bf 24} (2012), 99--120.

\bibitem{DW2} C. David and J. Wu, 
`Pseudoprime reductions of elliptic curves',
{\it Canadian J. Math.\/},  {\bf 64} (2012), 81--101.

\bibitem{Elk1} N. Elkies, `Distribution of supersingular primes',
{\it Ast\'erisque, J. Arithm\'etiques de Luminy,
1989\/}, {\bf 198-200} (1991), 127--132.

\bibitem{Elk2}
N. Elkies, `Elliptic and modular curves over finite fields and related computational issues',
{\it Computational perspectives on number theory}, D. A. Buell and J. T. Teitelbaum eds.,
{\it Studies in Advanced Mathematics\/}, Amer.  Math.  Soc.,
Providence, RI, {\bf 7} (1998), 21--76.

\bibitem{Fur} M. F\"urer, `Faster integer multiplication', {\it SIAM J. Comput.} {\bf 39} (2009), 979--1005.

\bibitem{Gal}
S. Galbraith, {\it Mathematics of public key cryptography},
Cambridge University Press, 2012.

\bibitem{GG}
J. von zur Gathen and J. Gerhard, {\it Modern computer algebra\/}, 2nd ed.,
Cambridge University Press, 2003.



\bibitem{Harv1} D. Harvey, `Counting points on hyperelliptic curves in average polynomial time', {\it Ann. of Math.\/}, {\bf 179} (2014),  783--803.

\bibitem{Harv2} D. Harvey, `Computing zeta functions of arithmetic schemes',
{\it preprint\/}, 2014 (available
at \url{http://arxiv.org/abs/1402.3439}).

\bibitem{HvdHL} D. Harvey, J. van der Hoeven, and G. Lecerf, `Even faster integer multiplication',
{\it preprint\/}, 2014 (available
at \url{http://arxiv.org/abs/1407.3360}).

\bibitem{HarvSuth} D. Harvey and A. Sutherland, `Computing Hasse--Witt matrices of hyperelliptic curves in average polynomial time',
{\it Proc. Algorithmic Number Theory Symposium (ANTS XI), 2014\/}, 
LMS J. Comput. Math. \textbf{17} (2014), 257--273.

\bibitem{HarvSuth2} D. Harvey and A. Sutherland, `Computing Hasse--Witt matrices of hyperelliptic curves in average polynomial time, II',
{\it preprint\/}, 2014 (available
at \url{http://arxiv.org/abs/1410.5222}).

%

\bibitem{LaOd}
J. C. Lagarias and  A. M. Odlyzko, `Effective versions of the
Chebotarev density theorem', {\it  Algebraic Number Fields\/}, Acad.
Press, NY, 1977, 409--464.

\bibitem{LaLiSo} Y. Lamzouri,  X. Li and K. Soundararajan, 
`Conditional bounds for the least quadratic non-residue and related problems', 
{\it Math. Comp.\/}, (to appear).

\bibitem{Mont} H. L. Montgomery, {\it Topics in multiplicative number theory\/}, 
Lect. Notes in Math., vol. 227, Spinger-Verlag, Berlin, 1971.

\bibitem{Rab} M. O. Rabin, {\it Probabilistic algorithms in finite fields\/}, SIAM J. Comput. {\bf 9} (1980), 273--280.

\bibitem{RS} K. Rubin and A. Silverberg, `Choosing the correct elliptic curve in the CM method', {\it Math. Comp.} {\bf 79} (2010), 545--561.

\bibitem{Sat} T. Satoh, `On $p$-adic point counting algorithms
for elliptic curves over finite fields', {\it Algorithmic Number Theory 5th International Symposium (ANTS V)\/}
Lect. Notes in Comp. Sci., Springer-Verlag,
Berlin, {\bf 2369} (2002), 43--66.

\bibitem{SS} A. Sch\"{o}nhage and V. Strassen,
`Schnelle {M}ultiplikation gro\ss{}er {Z}ahlen', {\it Computing\/},
{\bf 7} (1971), 281--292.

\bibitem{Sch1} R. Schoof, `Elliptic curves over finite fields and the computation of square roots mod $p$', {\it Math. Comp.\/}, {\bf 44} (1985), 483--494.

\bibitem{Sch2} R. Schoof, `Counting points on elliptic curves over finite fields',
{\it J.~Th\'{e}orie des Nombres de Bordeaux\/}, {\bf 7} (1995), 219--254.

\bibitem{Ser} J.-P. Serre, `Propri{\'e}t{\'e}s galoisiennes des points d'ordre fini des courbes elliptiques', {\it Invent. Math.\/} {\bf 15}
(1972),   259--331.

\bibitem{Shp} I. E. Shparlinski,  `On the product of small Elkies primes', 
{\it  Proc. Amer. Math. Soc.\/},  (to appear).

\bibitem{ShpSuth}  I. E. Shparlinski and A. V. Sutherland, 
`On the distribution of Atkin and Elkies primes',
{\it Found. Comp. Math.\/},  (to appear). 

\bibitem{Silv} J. H. Silverman, {\it The arithmetic of elliptic
curves\/}, 2nd ed.,
Springer, Dordrecht, 2009.

\bibitem{Suth1} A. V. Sutherland, `Identifying supersingular elliptic curves',
{\it LMS J.  Comp. and Math.\/}, {\bf 15} (2012), 317--325.

\bibitem{Suth2} A. V. Sutherland, `On the evaluation of modular polynomials', 
{\it Proc. Algorithmic Number Theory  Symposium (ANTS X), 2012\/}, 
Math. Scie. Publ., 2013, 531--555.

\end{thebibliography}
\end{document}